\documentclass[a4paper,UKenglish]{article}
\pdfoutput=1
 \usepackage{hyperref}
\usepackage[protrusion=true,expansion=true]{microtype}
\usepackage[absolute]{textpos}
\usepackage[utf8]{inputenc}
\usepackage{amsfonts, amsopn, amsmath, amsthm}
\usepackage{graphicx}
\usepackage{authblk}   

\usepackage{xcolor}
\definecolor{darkblue}{rgb}{0,0,0.45}
\definecolor{darkred}{rgb}{0.6,0,0}
\definecolor{darkgreen}{rgb}{0.13,0.5,0}
\hypersetup{colorlinks, linkcolor=darkblue, citecolor=darkgreen, urlcolor=darkblue}

\newtheorem{theorem}{Theorem}

\newtheorem{lemma}[theorem]{Lemma}

\newtheorem{observation}[theorem]{Observation}
\newtheorem{conjecture}[theorem]{Conjecture}
\newtheorem{question}[theorem]{Question}
\newcommand{\ch}{\text{ch}}
\newcommand{\brm}[1]{\operatorname{#1}}

\title{Flexibility of planar graphs without 4-cycles\thanks{This research is a part of projects that have received funding from the European Research Council (ERC) under the European Union's Horizon 2020 research and innovation programme under grant agreements No 714704.
This work was partially supported by student grants of Charles University GAUK 1277018 and SVV–2017–260452.
}
}

\author[1,2]{Tomáš Masařík}

\affil[1]{Faculty of Mathematics, Informatics and Mechanics, University of Warsaw, Poland}
\affil[2]{Department of Applied Mathematics, Charles University, Prague, Czech Republic}
\affil[ ]{\texttt{masarik@kam.mff.cuni.cz}}

\date{}

\begin{document}

\maketitle

\begin{textblock}{20}(0, 12.5)
\includegraphics[width=40px]{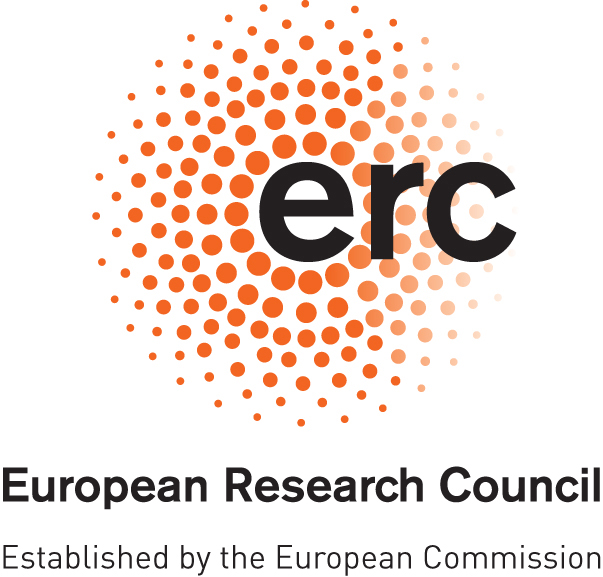}%
\end{textblock}
\begin{textblock}{20}(-0.25, 12.9)
\includegraphics[width=60px]{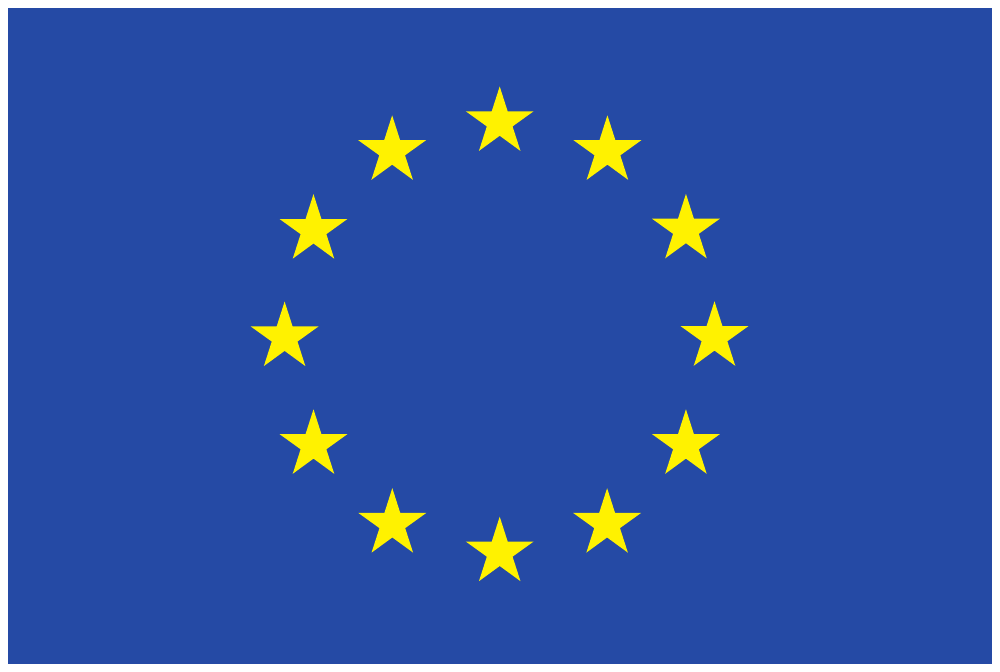}%
\end{textblock}

\begin{abstract}
Proper graph coloring assigns different colors to adjacent vertices of the graph.
Usually, the number of colors is fixed or as small as possible.
Consider applications (e.g.\ variants of scheduling) where colors represent limited resources and graph represents conflicts, i.e., two adjacent vertices cannot obtain the same resource.
In such applications, it is common that some vertices have preferred resource(s).
However, unfortunately, it is not usually possible to satisfy all such preferences.
The notion called flexibility was recently defined in [Dvořák, Norin, Postle: List coloring with requests, Journal of Graph Theory 2019].
There instead of satisfying all the preferences the aim is to satisfy at least a constant fraction of the request.

Recently, the structural properties of planar graphs in terms of flexibility were investigated.
We continue this line of research.
Let $G$ be a planar graph with a list assignment $L$.
Suppose a preferred color is given for some of the vertices.
We prove that if $G$ is a planar graph without 4-cycles and all lists have size at least five, then there exists an $L$-coloring respecting at least a constant fraction of the preferences.
\end{abstract}


\maketitle


\section{Introduction}

In a proper graph coloring, we want to assign to each vertex of a graph one of a fixed number of colors in such a way that adjacent vertices receive distinct colors.
Dvořák, Norin, and Postle~\cite{requests} (motivated by a similar notion considered by Dvořák and Sereni~\cite{dotri}) introduced the following graph coloring question called \emph{Flexibility}.
If some vertices of the graph have a preferred color, is it possible to properly color the graph so that at least a constant fraction of the preferences are satisfied?
As it turns out, this question is trivial in the ordinary proper coloring setting with a bounded number of colors ($k$-coloring).
The answer is always positive since we can permute the colors according to the request and therefore satisfy at least $1\over{k}$ fraction~\cite{requests}.
On the other hand, Flexibility brought about a number of interesting problems in the list coloring setting.

A \emph{list assignment} $L$ for a graph $G$ is a function that to each vertex $v\in V(G)$ assigns a set $L(v)$ of colors, and an \emph{$L$-coloring} is a proper coloring $\varphi$ such that $\varphi(v)\in L(v)$ for all $v\in V(G)$.
A graph $G$ is \emph{$k$-choosable} if $G$ is $L$-colorable from every assignment $L$ of lists of size at least $k$.
A \emph{weighted request} is a function $w$ that to each pair $(v,c)$ with $v\in V(G)$ and $c\in L(v)$ assigns a nonnegative real number.  Let $w(G,L)=\sum_{v\in V(G),c\in L(v)} w(v,c)$.
For $\varepsilon>0$, we say that $w$ is \emph{$\varepsilon$-satisfiable} if there exists an $L$-coloring $\varphi$ of $G$ such that
\[
 \sum_{v\in V(G)} w(v,\varphi(v))\ge\varepsilon\cdot w(G,L).
  \]

An important special case is when at most one color can be requested at each vertex and all such colors have the
same weight.
A \emph{request} for a graph $G$ with a list assignment $L$ is a function $r$ with $\brm{dom}(r)\subseteq V(G)$ such that $r(v)\in L(v)$ for all $v\in\brm{dom}(r)$.
For $\varepsilon>0$, a request $r$ is \emph{$\varepsilon$-satisfiable} if there exists an $L$-coloring $\varphi$ of $G$ such that $\varphi(v)=r(v)$ for at least $\varepsilon|\brm{dom}(r)|$ vertices $v\in\brm{dom}(r)$.

Note that in particular, a request $r$ is $1$-satisfiable if and only if the precoloring given by $r$ extends to an $L$-coloring of $G$.
We say that a graph $G$ with the list assignment $L$ is \emph{$\varepsilon$-flexible} if every request is $\varepsilon$-satisfiable, and it is \emph{weighted $\varepsilon$-flexible} if every weighted request is $\varepsilon$-satisfiable.

Dvořák, Norin, and Postle~\cite{requests} established the basic properties of the concept.
They prove several theorems in terms of degeneracy and maximum average degree.
For example: For every $d\ge 0$, there exists $\varepsilon>0$ such that $d$-degenerate graphs with assignment of lists of size $d+2$ are weighted $\varepsilon$-flexible.
Those results imply structural theorems for planar graphs:
\begin{itemize}
  \item There exists $\varepsilon>0$ such that every planar graph with an assignment of lists of size $6$ is $\varepsilon$-flexible.
  \item There exists $\varepsilon>0$ such that every planar graph of girth at least five with an assignment of lists of size $4$ is $\varepsilon$-flexible.
  \item There exists $\varepsilon>0$ such that every planar graph of girth at least 12 with an assignment of lists of size $5$ is weighted $\varepsilon$-flexible.
\end{itemize}
Those results prompted a number of interesting questions.
The main meta-question for planar graphs is whether such bounds can be improved to match the choosability.
Notice that choosability is a lower bound for the minimum size of lists in the statement.
Dvořák, Masařík, Musílek, and Pangrác subsequently answer two such questions.
In~\cite{req-trfree} they show that triangle-free planar graphs with an assignment of lists of size $4$ are weighted $\varepsilon$-flexible.
This is optimal since there are triangle-free planar graphs that are not 3-choosable~\cite{glebov,voigt1995}.
In~\cite{req-six} they show that planar graphs of girth at least six with an assignment of lists of size $3$ are weighted $\varepsilon$-flexible.
There is still a small gap left open since even planar graphs of girth at least 5 are 3-choosable\cite{voigt1993}.
The biggest question in this direction that is still unanswered is stated as follows.
\begin{question}\label{q:general}
Does there exist $\varepsilon>0$ such that every planar graph $G$ and assignment $L$ of lists of size five is (weighted) $\varepsilon$-flexible?
\end{question}
This would be optimal in terms of choosability~\cite{voigt1993,thomassen1995-34}.
However, (if it is true) it might be difficult to obtain such a result since even the result of Thomassen~\cite{thomassen1995-34} for choosability is very involved.
In particular, compare it to a rather easy proof~\cite{Kratochvil} for choosability of triangle-free planar graphs and still the respective result for flexibility~\cite{req-trfree} was quite technical.

In this paper, we propose a step towards answering Question~\ref{q:general} by proving the following theorem.
\begin{theorem}\label{main}
  There exists $\varepsilon>0$ such that each planar graph without 4-cycles with an assignment of lists of size five is weighted $\varepsilon$-flexible.
\end{theorem}

Since planar graphs without 4-cycles are 4-choosable~\cite{chsq-free} there is a gap left open.

\section{Preliminaries}
We say that a face is \emph{edge-adjacent} to another face if both share an edge.
Since graphs we are dealing with does not contain 4-cycles they cannot contain two edge-adjacent triangles.

Let $H$ be a graph.
For a positive integer $d$, a set $I\subseteq V(H)$ is \emph{$d$-independent} if the distance between any distinct vertices of $I$ in $H$ is greater than $d$.
Let $1_I$ denote the characteristic function of $I$, i.e., $1_I(v)=1$ if $v\in I$ and $1_I(v)=0$ otherwise.
For functions that assign integers to vertices of $H$, we define addition and subtraction in a natural way, adding/subtracting their values at each vertex independently.
For a function $f:V(H)\to\mathbb{Z}$ and a vertex $v\in V(H)$, let $f\downarrow	v$ denote the function such that $(f\downarrow v)(w)=f(w)$ for $w\neq v$ and $(f\downarrow v)(v)=1$.
A list assignment $L$ is an \emph{$f$-assignment} if $|L(v)|\ge f(v)$ for all $v\in V(H)$.

Suppose $H$ is an induced subgraph of another graph $G$.
For an integer $k\ge 3$, let $\delta_{G,k}:V(H)\to\mathbb{Z}$
be defined by $\delta_{G,k}(v)=k-\deg_G(v)$ for each $v\in V(H)$.
For another integer $d\ge 0$, we say that $H$ is a \emph{$(d,k)$-reducible} induced subgraph of $G$ if
\begin{itemize}
\item[(FIX)] for every $v\in V(H)$, $H$ is $L$-colorable for every $((\deg_H+\delta_{G,k})\downarrow v)$-assignment $L$, and
\item[(FORB)] for every $d$-independent set $I$ in $H$ of size at most $k-2$, $H$ is $L$-colorable for every $(\deg_H+\delta_{G,k}-1_I)$-assignment $L$.
\end{itemize}
Note that (FORB) in particular implies that $\deg_H(v)+\delta_{G,k}(v)\ge 2$ for all $v\in V(H)$.
Intuitively, (FIX) requires that $H$ is $L'$-colorable even if we prescribe the color of any single vertex of $H$,
and (FORB) requires that $H$ is $L'$-colorable even if we forbid to use one of the colors on the set $I$.

The general version of the following lemma is implicit in Dvo\v{r}\'ak et al.~\cite{requests} and appears explicitly in~\cite{req-trfree}.

\begin{lemma}\label{lemma-redu-simpl}
For all integers $b\ge 1$, there exists $\varepsilon>0$ as follows.
If for every $Z\subseteq V(G)$, the graph $G[Z]$ contains an induced $(0,5)$-reducible subgraph with at most $b$ vertices, then $G$ with any assignment of lists of size $5$ is weighted $\varepsilon$-flexible.
\end{lemma}

\section{Reducible configurations}

In view of Lemma~\ref{lemma-redu-simpl}, we aim to prove that every planar graph without 4-cycles contains a $(0,5)$-reducible induced subgraph with the bounded number of vertices.

\begin{observation}\label{obs-deg1}
  In any graph $G$, a vertex of degree at most $3$ forms a $(0,5)$-reducible subgraph.
\end{observation}

From now on suppose that the minimum degree of $G$ is $4$.
We describe one more easy reducible configuration (see Figure~\ref{fig:redu}) that, in combination with discharging, turns out to be sufficient to derive the promised theorem.

\begin{lemma}\label{lem:redu}
  If $G$ is a planar graph without 4-cycles, then a vertex $v$ together with $\deg(v)-1$ neighbors of degree four forms a $(0,5)$-reducible configuration on $\deg(v)$ vertices.
\end{lemma}

\begin{proof}
  (FIX): If vertex $v$ has fixed color then we have enough remaining colors on its neighbors to complete the coloring.
  If any other vertex $v'$ is fixed then it crosses out one color from $v$ and in case $v'$ has a neighbor in $N(v)$ it also crosses out one of its colors. In both cases, we can set a color of $v$ and complete the coloring greedily.

  (FORB): Observe that if we forbid a color of a vertex $v'$ that is not adjacent to any vertex in $N(v)$ then its color is determined and therefore it crosses out one color of $v$.
  The same effect has a forbidden color of $v$.
  If we forbid a color of a vertex $v_1$ such that it forms a triangle $v,v_1,v_2$ then it does not force anything unless vertex $v_2$ has also a forbidden color.
  In the latter case two colors are crossed out from the list of $v$.
  Keep in mind that this cannot happen twice since there are no two edge-adjacent triangles.
  Since only three colors are removed from the list of $v$, we can color $v$ and then the rest of the graph greedily to conclude the proof.
\end{proof}

 \begin{figure}
  \centering
   \includegraphics{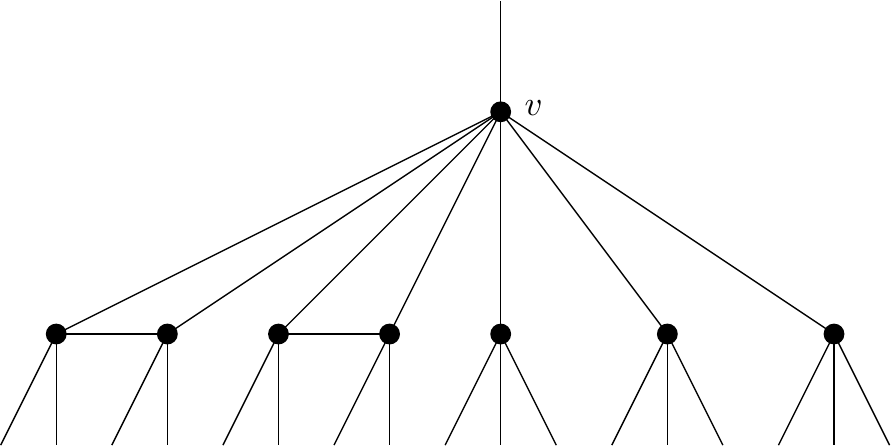}
   \caption{A reducible configuration in Lemma~\ref{lem:redu}.}\label{fig:redu}
 \end{figure}

\section{Discharging}

Let us assign charge $\ch_0(v)=\deg(v) - 4$ to each vertex $v\in V(G)\setminus V(C)$ and charge $\ch_0(f)=|f|-4$ to each face $f$ of $G$, where $|f|$ denotes the length of the facial walk of $f$.
By Euler's formula, we have $\sum_{v\in V(G)} \ch_0(v)+\sum_{f\in F(G)}\ch_0(f)=(2|E(G)|-4|V(G)|)+(2|E(G)|-4|F(G)|)=4(|E(G)|-|V(G)|-|F(G)|)=-8$.

Note that only triangle-faces have a negative charge before any redistribution of the charge.
We redistribute the initial charge according to the following rules.
\begin{itemize}
  \item[(R1)] For each face $f$ of $G$ if $|f|\ge5$ then $f$ sends $1\over5$ to any edge-adjacent triangle-face.
  \item[(R2)] For each vertex $v$ of $G$ if $\deg(v)\ge 5$ then $v$ sends $2\over5$ to each adjacent triangle-face.
  \item[(R3)] For each vertex $v$ of $G$ and each its incident face $f$ if $\deg(v)\ge 5$ and $|f|\ge5$ then $v$ sends $1\over 15$ to $f$.
\end{itemize}

Observe that the charge of any face $|f|\ge 5$ does not drop below zero by Rule (R1).
Any vertex $v$ of degree at least 5 with $t$ incident triangle-faces sends at most ${2t}\over{5}$ by Rule (R2) and ${(\deg(v)-t)}\over{15}$ by Rule (R3).
View that the number of incident triangle-faces for a single vertex is at most $\lfloor\frac{\deg(v)}{2}\rfloor$ because there are not any edge-adjacent triangle-faces.
Therefore $t\le\lfloor\frac{\deg(v)}{2}\rfloor$.
It follows that $\ch(v)\ge 0$.

It remains to argue that all triangle-faces obtain enough of the charge.
Each of them receives the charge at least $3\over5$ by Rule (R1).
If one of its vertices has degree at least five we are done by Rule (R2).
Therefore all of them have degree exactly four.
We call such triangle-face \emph{poor}.

We do one more redistribution of charge.

\begin{itemize}
  \item[(R4)] For each poor triangle-face $f$ of $G$ and for each edge-adjacent face $f'$ if $|f'|\ge5$ then $f'$ sends ${2}\over{15}$ to $f$.
\end{itemize}

A Combination of Rules (1) and (4) yields that poor triangle-faces have a positive charge.
They obtain $3\over5$ by Rule (1) and three times $2\over{15}$ by Rule (4).
Finally, we show that the charge of larger faces remains non-negative after the application of Rule (R4).

Consider face $f=v_1,\ldots v_k$ of length at least five edge-adjacent to some triangle-face $f_t=v_1,v_2,v'$.
Recall that $\deg(v_1)=\deg(v_2)=\deg(v')=4$.
By Lemma~\ref{lem:redu} applied on vertex $v_1$ we claim that $\deg(v_k)\ge5$.
By the same argument repeated on vertex $v_2$ we derive $\deg(v_3)\ge5$.
Therefore, by Rule (R3) $f$ receive charge at least $2\over{15}$ from each $v_3$ and $v_k$.
This, combined with an observation that $f$ has at most $|f|-4$ edge-adjacent poor triangle-faces, yields the promised claim for $|f|=5$.
Larger faces sent only ${|f|\over5}$ by Rule (R1) altogether and therefore they can pay an additional ${2(|f|-5)}\over{15}$.
\[
{{|f|}\over5}+ {{2(|f|-5)}\over{15}}=
{{|f|-2}\over{3}}
\le |f|-4.
\]

This is a contradiction with the original negative assignment of charge and therefore we derive Theorem~\ref{main}.

\section{Conclusions}
We proved that planar graphs without 4-cycles are weighted $\varepsilon$-flexible for lists of size at least five.
This is a middle step to answer Question~\ref{q:general} that might be challenging as mentioned in the introduction.
Based on the proof possible difficulties we suggest, as a next step to prove the conjecture, to inspect first planar graphs without diamonds ($K_4-e$).

\begin{conjecture}\label{q:new}
There exists $\varepsilon>0$ such that every planar graph $G$ without diamonds and assignment $L$ of lists of size five is (weighted) $\varepsilon$-flexible.
\end{conjecture}

Another possible direction is closing the gap between flexibility and choosability for planar graphs without 4-cycles.

\begin{question}\label{q:c4}
Does there exists $\varepsilon>0$ such that every planar graph $G$ without 4-cycles and assignment $L$ of lists of size four is (weighted) $\varepsilon$-flexible?
\end{question}

\section{Acknowledgments}
  I would like to thank Zdeněk Dvořák for helpful comments.

\end{document}